\documentclass[12pt, reqno]{amsart}
\usepackage{amsmath, amsthm, amscd, amsfonts, amssymb, graphicx, xcolor}
\usepackage[bookmarksnumbered, colorlinks, plainpages]{hyperref}
\usepackage{amsmath}
\usepackage{cases}
\usepackage{amssymb}
\usepackage{mathrsfs}
\usepackage{euscript}
\usepackage{dsfont}
\usepackage{graphicx}
\usepackage{float}
\allowdisplaybreaks[4]
\newtheorem{theorem}{Theorem}[section]

\newtheorem{lemma}{Lemma}[section]
\newtheorem{proposition}[theorem]{Proposition}

\theoremstyle{definition}

\newtheorem{conjecture}[theorem]{Conjecture}

\theoremstyle{remark}
\newtheorem{remark}[theorem]{Remark}
\numberwithin{equation}{section}

\begin{document}
	\setcounter{page}{1}
	

	
	
	\title{Note on the largest prime factors of consecutive integers}

	\author[]{Zhiyuan Yang}
	
	\address{School of Mathematics,  Shandong University, Jinan 250100, Shandong, China}
	\email{\tt zhiyuan.yang@mail.sdu.edu.cn}
	
	
	

	
	
	\begin{abstract}
		In this paper, we give an alternative proof of Tao-Teräväinen's results by showing
		\begin{align*}
			&\#\{n\leq x:P^+(n-1)<P^+(n)<P^+(n+1)>\max_iP^+(n+a_i)\}\gg x,\\&\#\{n\leq x:P^+(n-1)>P^+(n)<P^+(n+1)>\max_iP^+(n+a_i)\}\gg x,
		\end{align*}
		where the $a_i$ are at most finite in number, with $|a_i| = O(1)$. We also obtain some comparison results for five numbers under the Elliott-Halberstam conjecture for friable integers. For any $\varepsilon>0$, we also show
		\begin{align*}
			&\#\left\{n\leq x:1\leq \frac{P^+(n+1)}{P^+(n)}<1+\varepsilon\right\}\gg_\varepsilon \frac{x}{\log x},
		\end{align*}
		which improves a previous result $\gg x(\log x)^{-50}$ due to S. Kim (2020).
		\newline
		\noindent 
	\end{abstract} \maketitle
	

	\section{Introduction}
	
	Let $P^+(n)$ denote the largest prime factor of $n$ with the convention that $P^+(1)=1$. In the 1930s, Erd\H{o}s and Turán [14] formulated the following conjecture in a correspondence.
	\begin{conjecture}
		(Erd\H{o}s-Turán) For $x\rightarrow\infty$, we have
		\begin{align*}
			\#\{n\leq x:P^+(n)<P^+(n+1)\}\sim \frac{1}{2}x.
		\end{align*}
	\end{conjecture}
	In 1978, Erd\H{o}s and Pomerance [4] first proved that there exists a positive asymptotic density of integers $n$ such that $P^+(n)<P^+(n+1)$. In fact, they showed that\begin{align*}
		\#\{n\leq x:P^+(n)<P^+(n+1)\}>0.0099x.
	\end{align*}
	The best published lower constant is 0.2017 due to Lü-Wang [12]. Very recently, the author's preprints [18, 19] showed that the asymptotic density is larger than 0.280 and 0.299. See [12] for more historical background.
	
	We now consider the case of comparing more than two numbers. In 2011, De Koninck and Doyon [2] formulated a more general conjecture.
	\begin{conjecture}
		(De Koninck-Doyon) Fix an arbitrary integer $k\geq 2$ and let $n$ be a large number. Let $a_1,a_2,\cdots,a_k$ be any permutation of the numbers $0,1,\cdots,k-1$. Then we have\begin{align*}
			\#\{n\leq x:P^+(n+a_1)<P^+(n+a_2)<\cdots<P^+(n+a_k)\}\sim\frac{1}{k!}x.
		\end{align*}
	\end{conjecture}
	There has also been some recent progress for this type of comparison problem. 
	In 1978, Erd\H{o}s and Pomerance [4] observed that
	the patterns
	\begin{align*}
		P^+(n-1)>P^+(n)<P^+(n+1),\ \text{or}\ P^+(n-1)<P^+(n)>P^+(n+1)
	\end{align*}
	and 
	\begin{align*}
		P^+(n-1)<P^+(n)<P^+(n+1)
	\end{align*}
	both occur infinitely often and conjectured that these two patterns occur for a
	positive density of $n$. In addition, they proved that there are infinitely many such integers. In 2001, Balog [1] proved the following lower bound
	\begin{align*}
		\#\{n\leq x:P^+(n-1)<P^+(n)<P^+(n+1)\}\gg x^{1/2}.
	\end{align*}
	Later, in 2018, Wang [16] obtained two positive patterns by showing
	\begin{align*}
		&\#\{n\leq x:P^+(n-1)>P^+(n)<P^+(n+1)\}\gg x,\\&\#\{n\leq x:P^+(n-1)<P^+(n)>P^+(n+1)\}\gg x.
	\end{align*}
	For more difficult cases, Tao and Teräväinen [15] succeeded in obtaining a positive
	density by showing
	\begin{align*}
		&\#\{n\leq x:P^+(n-1)<P^+(n)<P^+(n+1)>P^+(n+2)\}\gg x,\\&\#\{n\leq x:P^+(n-1)>P^+(n)>P^+(n+1)<P^+(n+2)\}\gg x.
	\end{align*}
	However, Tao-Teräväinen's proof is very complicated.
	In particular, using estimates on the numbers of prime $k$-tuples,
	Kim [10] proved the following result
	\begin{align*}
		\#\{n\leq x:P^+(n-1)<P^+(n)<P^+(n+1)\}\gg \frac{x}{(\log x)^{50}}.
	\end{align*}
	
	In this paper, motivated by [3], we first give an alternative proof of Tao-Teräväinen's results.
	\begin{theorem}
		For sufficiently large $x$, we have
		\begin{align*}
			&\#\{n\leq x:P^+(n-1)<P^+(n)<P^+(n+1)>\max_iP^+(n+a_i)\}\gg x,\\&\#\{n\leq x:P^+(n-1)>P^+(n)<P^+(n+1)>\max_iP^+(n+a_i)\}\gg x,
		\end{align*}
		where the $a_i$ are at most finite in number, with $|a_i| = O(1)$.
	\end{theorem}
	\begin{remark}
		The result we establish in Theorem 1.3 may be stronger than Tao-Teräväinen’s. However, in [15], they do not give the type of the implied constant.
	\end{remark}
	
	Our second theorem requires the Elliott-Halberstam conjecture for friable integers, which means that for any $\varepsilon$ and any $A>0$, one has
	\begin{align*}
		\sum_{q\leq x^{1-\varepsilon}}\max_{x'\leq x}\max_{(a,q)=1}\Bigg|\sum_{\substack{n\leq x',P^+(n)\leq y\\n\equiv a\mkern-15mu\pmod{q}}}1-\frac{1}{\varphi(q)}\sum_{\substack{n\leq x',P^+(n)\leq y\\(n,q)=1}}1\Bigg|\ll_A\frac{x}{(\log x)^A},
	\end{align*}
	where $x^\varepsilon<y\leq x$.
	\begin{theorem}
		Under the Elliott-Halberstam conjecture for friable integers, we have
		\begin{align*}
			&\#\{n\leq x:P^+(n-1)>P^+(n)<P^+(n+1)<P^+(n+2)>P^+(n+3)\}\gg x,\\&\#\{n\leq x:P^+(n-1)<P^+(n)>P^+(n+1)>P^+(n+2)<P^+(n+3)\}\gg x,\\&\#\{n\leq x:P^+(n-1)>P^+(n)<P^+(n+1)>P^+(n+2)<P^+(n+3)\}\gg x,\\&\#\{n\leq x:P^+(n-1)<P^+(n)>P^+(n+1)<P^+(n+2)>P^+(n+3)\}\gg x.
		\end{align*}
	\end{theorem}
	\begin{remark}
		We would prefer to obtain comparisons for four consecutive numbers with the same sign $<$ (or $>$).
	\end{remark}
	Furthermore, in Theorem 2 of [10], Kim showed that
	\begin{align*}
		\#\left\{n\leq x:1\leq \frac{P^+(n+1)}{P^+(n)}<1+\varepsilon\right\}\gg_\varepsilon \frac{x}{(\log x)^{50}}.
	\end{align*}
	In this paper, we establish a better lower bound.
	\begin{theorem}
		For any $\varepsilon>0$, we have
		\begin{align*}
			&\#\left\{n\leq x:1\leq \frac{P^+(n+1)}{P^+(n)}<1+\varepsilon\right\}\gg_\varepsilon \frac{x}{\log x},\\&\#\left\{n\leq x:1-\varepsilon\leq \frac{P^+(n+1)}{P^+(n)}<1\right\}\gg_\varepsilon \frac{x}{\log x}.
		\end{align*}
	\end{theorem}
	\begin{remark}
		By a heuristic argument, we may conjecture that its order is $\asymp_\varepsilon x(\log x)^{-1}$.
	\end{remark}

	\section{Lemmas}
	
	\begin{lemma}
		Let $g$ be a natural number and let $a_i, b_i $ ($i=1,2,\cdots, g$) be integers satisfying
		\begin{align*}
			E:=\prod_{i=1}^g a_i\prod_{1\leq r<\leq g}(a_r b_s-a_sb_r)\neq 0.
		\end{align*}
		Let $\omega(p)$ denote the number of solutions in n modulo $p$ of 
		\begin{align*}
			\prod_{i=1}^g(a_in+b_i)\equiv 0\mkern-15mu\pmod{p},
		\end{align*}
		and suppose that
		\begin{align*}
			\omega(p)<p, \ \text{for all}\ p.
		\end{align*}
		If the real numbers $y$ and $z$ satisfy $1<y\leq z$, then 
		\begin{align*}
			&\#\{n:z-y<n\leq z, a_in+b_i \ \text{prime for}\ i=1,2,\cdots, g\}\\&\leq 2^gg!\prod_p\left(1-\frac{\omega(p)-1}{p-1}\right)\left(1-\frac{1}{p}\right)^{-g+1}\frac{y}{\log^g y}\Bigg(1+O\left(\frac{\log\log (3y)+\log\log |3E|}{\log y}\right)\Bigg),
		\end{align*}
		where the constant implied by the $O$-symbol depends at most on $g$.
	\end{lemma}
	\begin{proof} 
		This is [7, Theorem 5.7]. It is derived from Selberg's upper-bound sieve (see Lemma 2.3 below).
	\end{proof}
	
		\begin{lemma}
		Let $D\geq 2$ and $L>1$. Let $\mathscr{P}$ denote a set of primes. Let $z\geq 2$ and write $P(z):=\prod_{p\leq z,p\in\mathscr{P}}p$. There exist two sequences $\{\lambda_d^\pm\}_{d=1}^{\infty}$ of real numbers, vanishing for $d>D$ or $\mu(d)=0$, satisfying $\lambda_1^\pm=1$, $|\lambda_d^{\pm}|\leq 1$, $\lambda^-*1\leq \mu*1\leq \lambda^+*1$, and such that
		\begin{align*}
			&\sum_{d|P(z)}\lambda_d^+\frac{\omega(d)}{d}\leq \prod_{\substack{p\leq z\\p\in \mathscr{P}}}\left(1-\frac{\omega(p)}{p}\right)\left(F(s)+O\left(\frac{e^{L^{1/2}-s}}{(\log D)^{1/3}}\right)\right),\\&\sum_{d|P(z)}\lambda_d^-\frac{\omega(d)}{d}\geq \prod_{\substack{p\leq z\\p\in \mathscr{P}}}\left(1-\frac{\omega(p)}{p}\right)\left(f(s)+O\left(\frac{e^{L^{1/2}-s}}{(\log D)^{1/3}}\right)\right)
		\end{align*}
		uniformly for all multiplicative function $\omega$ satisfying
		\begin{align*}
			&(i)\ 0<\omega(p)<p\ (p\in\mathscr{P}),\\&(ii)\ \prod_{\substack{u<p\leq v\\p\in\mathscr{P}}}\left(1-\frac{\omega(p)}{p}\right)^{-1}\leq \frac{\log v}{\log u}\left(1+\frac{L}{\log u}\right)\ (2\leq u\leq v\leq z),
		\end{align*}
		where $s=\log D/\log z$ and $f(s)$, $F(s)$ are determined by the differential-difference equations
		\begin{equation}
			\begin{aligned}
				&\left\{ \begin{aligned}
					&sF(s)=2e^{\gamma}, \quad&& 1\leq s\leq 2;\\&sf(s)=0,\quad&& 0\leq s\leq 2;\\&(sF(s))'=f(s-1),(sf(s))'=F(s-1),\quad&&s\geq 2.
				\end{aligned}\right.\nonumber
			\end{aligned}
		\end{equation}
	\end{lemma}
	\begin{proof}
		This lemma is the Rosser-Iwaniec sieve [9].
	\end{proof}

	\begin{lemma}[Selberg’s $\Lambda^2$-sieve] Let $\xi\geq 2$ and $L>1$. Let $\mathscr{P}$ denote a set of primes. Let $\nu(n)$ define the number of distinct prime factors of $n$. Let $z\geq 2$ and write $P(z):=\prod_{p\leq z,p\in\mathscr{P}}p$. There exists a sequence $\{\lambda_d\}_{d=1}^{\infty}$ of real numbers, vanishing for $d>\xi$ or $\mu(d)=0$, satisfying $\lambda_1=1$, $|\lambda_d|\leq 3^{\nu(d)}$, $\mu*1\leq \lambda*1$, and such that
		\begin{align*}
			\sum_{d|P(z)}\lambda_d\frac{\omega(d)}{d}\leq\frac{1}{G(\xi,z)}= \prod_{\substack{p\leq z\\p\in \mathscr{P}}}\left(1-\frac{\omega(p)}{p}\right)\left(\frac{1}{\sigma(s)}+O\left(\frac{Ls^5}{\log z}\right)\right),
		\end{align*}
		uniformly for all multiplicative function $\omega$ satisfying
		\begin{align*}
			&(i)\ 0<\frac{\omega(p)}{p}\leq 1-\frac{1}{A_1}\ (p\in\mathscr{P}),\\&(ii)-L\leq\sum_{w\leq p<z}\frac{\omega(p)}{p}\log p-2\log \frac{z}{w}\leq A_2\  (2\leq w\leq z),
		\end{align*}
		where $A_1\geq 1$, $A_2\geq 1$ are some suitable constants, $s=\log \xi/\log z$ and $\sigma(s)$ is determined by the differential-difference equations
		\begin{equation}
			\begin{aligned}
				&\left\{ \begin{aligned}
					&s^{-2}\sigma(s)=8^{-1}e^{-2\gamma},\quad&& 0\leq s\leq 2,\\&s(\sigma(s))'=2\sigma(s)-2\sigma(s-2),\quad&&s> 2.
				\end{aligned}\right.\nonumber
			\end{aligned}
		\end{equation}
	\end{lemma}
	\begin{proof}
		This is Selberg’s $\Lambda^2$-sieve. We can see it in [7, Lemma 6.1].
	\end{proof}

	\begin{lemma}
		For $\varepsilon>0$ and $a\in\{1,-1\}$, we have
		\begin{align*}
			\frac{1}{\log x}\sum_{p'>x^{1/2}}\log p'\sum_{p>L}\sum_{\substack{l\leq x/p\\lp\equiv a\mkern-15mu\pmod{p'}}}1=\frac{x}{2}\log\frac{\log x}{\log L}+o(x),
		\end{align*}
		where $x^\varepsilon<L<x^{1-\varepsilon}$.
	\end{lemma}
	\begin{proof}
		This is easy from [6] and Lemma 2.6 below.
	\end{proof}
	
	\begin{lemma}
		For $\varepsilon>0$, we have
		\begin{align*}
			\sum_{n\leq x,P^+(n)\leq y}1=x\rho(u)\left(1+O_\varepsilon\left(\frac{\log(u+1)}{\log y}\right)\right)
		\end{align*}
		uniformly for
		\begin{align*}
			x\geq x_0(\varepsilon),\qquad \exp{(\log_2x)^{5/3+\varepsilon}}\leq y\leq x,
		\end{align*}
		where $u=\log x/\log y$ and $\rho(u)$ is the Dickman-de Bruijn function which is defined by differential equation
		\begin{equation}
			\begin{aligned}
				&\left\{ \begin{aligned}
					&\rho(u)=1, \quad&& 0\leq u\leq 1,\\&u\rho'(u)=-\rho(u-1)\quad&& u>1.
				\end{aligned}\right.\nonumber
			\end{aligned}
		\end{equation}
	\end{lemma}
	\begin{proof}
		This is [8, Theorem 1].
	\end{proof}
	
	\begin{lemma}[Generalized Bombieri–Vinogradov theorem]
		Let $MN = x$, $x^\varepsilon< M < x^{1-\varepsilon}$ and fix
		$A > 0$. Let $(\alpha_m)$, $(\beta_n)$ be divisor-bounded sequences supported on $[M, 2M]$ and $[N, 2N]$, respectively,
		with $(\beta_n)$ satisfying the Siegel–Walfisz condition, which means
		\begin{align*}
			\Bigg|\sum_{\substack{n\sim N\\n\equiv b\mkern-15mu\pmod{q}\\(n,d)=1}}\beta_n-\frac{1}{\varphi(q)}\sum_{\substack{n\sim N\\(n,dq)=1}}\beta_n\Bigg|\ll_A\frac{N\tau(d)^{O(1)}}{(\log N)^A},
		\end{align*}
		for any $d\geq 1, q\geq 1$ and $(b,q)=1$, $A>1$.
		Then for some sufficiently large $B$ and $Q\leq x^{1/2}(\log x)^{-B}$, we have
		\begin{align*}
			\sum_{q\leq Q}\max_{(a,q)=1}\Bigg|\sum_{\substack{m,n\\mn\equiv a\mkern-15mu\pmod{q}}}\alpha_m\beta_n-\frac{1}{\varphi(q)}\sum_{\substack{m,n\\(mn,q)=1}}\alpha_m\beta_n\Bigg|\ll\frac{x}{(\log x)^A}.
		\end{align*}
	\end{lemma}
	\begin{proof}
		We can see it in [5, Theorem 9.16].
	\end{proof}

	Define
	\begin{align*}
		\Psi(x,y):=\sum_{n\leq x,P^+(n)\leq y}1,\ \Psi_q(x,y):=\sum_{\substack{n\leq x,P^+(n)\leq y\\(n,q)=1}}1,\ \Psi(x,y;b,q):=\sum_{\substack{n\leq x,P^+(n)\leq y\\n\equiv b\mkern-15mu\pmod{q}}}1.
	\end{align*}
	\begin{lemma}
		Let $b\in\mathbb{Z}^*$ and $A$, $\varepsilon>0$. Then there exists $C=C(A,\varepsilon)>0$ such that, in the range $x>2$, $(\log x)^C\leq y\leq x^{1/C}$, one has
		\begin{align*}
			\sum_{\substack{q\leq x^{66/107-\varepsilon}\\(q,b)=1}}\Big|\Psi(x,y;b,q)-\frac{\Psi_q(x,y)}{\varphi(q)}\Big|\ll_{b,A,\varepsilon}\frac{\Psi(x,y)}{(\log x)^A}.
		\end{align*}
	\end{lemma}
	\begin{proof}
		This lemma is [13, Theorem 1.1].
	\end{proof}
	
	\subsection*{Convention.} We use $\varepsilon$ to denote a sufficiently small positive number, and the value of $\varepsilon$ may change from statement to statement. We use the standard asymptotic notation $f\ll g,f\gg g,f=O(g),f=o_{x\rightarrow\infty}(g)$ from analytic number theory. Unless otherwise specified, the constants implicit in $\ll$ and $\gg$ are absolute (independent of all parameters like $\varepsilon, c, u, \delta, \eta$).
	The constants implicit in the notation $\ll_\varepsilon,\gg_\varepsilon$ depend at most on $\varepsilon$.
	In our proof, unless otherwise specified, $C$, $C_1$, and $C_2$ denote suitable absolute constants (independent of all parameters).

	\section{Proof of Theorem 1.3}
	Before starting our proof, we need the following result.
	\begin{proposition}
		For every $\varepsilon>0$, there exists $\eta=\eta(\varepsilon)>0$ and $\delta=\delta(\eta,\varepsilon)>0$ such that for any $0<c<\eta$, the set $\mathcal{C}_{c,\delta}$ defined by
		\begin{align*}
			&\mathcal{C}_{c,\delta}:=\\
			&\left\{n\leq x:P^+(n)>x^{1-c}, P^+(n-1)>x^{1/2},P^+(n+1)>x^{1/2},x^{-\delta}<\frac{P^+(n-1)}{P^+(n+1)}<x^\delta\right\}
		\end{align*}
		satisfies
		\begin{align*}
			\mathcal{C}_{c,\delta}\leq\varepsilon cx.
		\end{align*}
	\end{proposition}
	\begin{proof}
		The proof is similar to that of [3, Theorem 3]. 	Its proof requires involved sieve calculations (by Lemma 2.1 and Lemma 2.2) but presents no essential difficulties. We only present the proof sketch. Fix $\varepsilon$ and $\eta$. For any $0<c<\eta$, we can write
		\begin{align*}
			\mathcal{C}_{c,\delta}&\leq \sum_{\substack{n\leq x\\P^+(n)>x^{1-c}\\x^{1/2}<P^+(n-1)<x^{1/2+\eta}\ \text{or}\ x^{1/2}<P^+(n+1)<x^{1/2+\eta}}}1+\sum_{\substack{n\leq x\\P^+(n)>x^{1-c}\\\min(P^+(n-1),P^+(n+1))>x^{1/2+\eta}\\x^{-\delta}<P^+(n-1)/P^+(n+1)<x^\delta}}1\\&=\mathcal{C}_1+\mathcal{C}_2.
		\end{align*}
		For $\mathcal{C}_1$, we have
		\begin{align*}
			\mathcal{C}_1&\leq\sum_{\substack{n\leq x\\P^+(n)>x^{1-c}\\x^{1/2}<P^+(n-1)<x^{1/2+\eta}}}1+\sum_{\substack{n\leq x\\P^+(n)>x^{1-c}\\x^{1/2}<P^+(n+1)<x^{1/2+\eta}}}1\\&\ll \eta cx.
		\end{align*}
		For $\mathcal{C}_2$, we rewrite $n=lp$, $n-1=ap_1$ and $n+1=bp_2$, where $p>x^{1-c}$ and $\min(p_1,p_2)>x^{1/2+\eta}$. Then we have $abl\ll x^{1-2\eta+c}\ll x^{1-\eta}$. Thus
		\begin{align*}
			\mathcal{C}_2\ll C_\eta\delta cx,
		\end{align*}
		for a suitable constant $C_\eta$.
		Now we first choose a suitable fixed $\eta=\eta(\varepsilon)>0$. For any $0<c<\eta$, we then choose a fixed $\delta=\delta(\eta,\varepsilon)>0$ such that
		\begin{align*}
			\mathcal{C}_{c,\delta}\leq \varepsilon cx.
		\end{align*}
	\end{proof}
	\begin{remark}
		In the proposition above, our main requirement is that $\delta$ does not depend on $c$, or that $c$ does not depend on $\delta$.
	\end{remark}
	We now return to our proof. By Proposition 3.1 with $\varepsilon=1/4$, there exists $\eta>0$ and $\delta=\delta(\eta)>0$ such that for any $0<c<\eta$, the set $\mathcal{C}_\delta$ defined by
	\begin{align*}
		\#\mathcal{C}_{c,\delta}\leq\frac{cx}{4}\leq \frac{x}{4}\log\frac{1}{1-c}.
	\end{align*}
	Now we can write
	\begin{align*}
		\sum_{\substack{n\leq x\\P^+(n)>x^{1-c}\\P^+(n-1)>x^{1/2}}}1&= \sum_{\substack{n\leq x\\P^+(n)>x^{1-c}\\P^+(n-1)>x^{1/2}\\P^+(n-1)<P^+(n+1)}}1+\sum_{\substack{n\leq x\\P^+(n)>x^{1-c}\\P^+(n-1)>x^{1/2}\\P^+(n-1)>P^+(n+1)}}1\\&\leq \sum_{\substack{n\leq x\\P^+(n)>x^{1-c}\\P^+(n-1)<P^+(n+1)}}1+\sum_{\substack{n\leq x, n\notin\mathcal{C}_{c,\delta}\\P^+(n)>x^{1-c}\\P^+(n-1)>x^{1/2}\\P^+(n-1)>P^+(n+1)}}1+\#\mathcal{C}_{c,\delta}\\&:=\mathcal{B}_1+\mathcal{B}_2+\#\mathcal{C}_{c,\delta}.
	\end{align*}
	Then by Lemma 2.4, we have
	\begin{align*}
		\mathcal{B}_1+\mathcal{B}_2\geq \sum_{\substack{n\leq x\\P^+(n)>x^{1-c}\\P^+(n-1)>x^{1/2}}}1-\#\mathcal{C}_{c,\delta}\geq \frac{x}{4}\log\frac{1}{1-c}.\tag{3.1}
	\end{align*}
	We now proceed to prove that $\mathcal{B}_2 \ll \mathcal{B}_1$. We define
	\begin{align*}
		U_n:=1_{P^+(n-1)>x^{1/2}}\log P^+(n-1),\quad V_n:=1_{P^+(n+1)>x^{1/2}}\log P^+(n+1).
	\end{align*}
	We also define
	\begin{align*}
		&P_c:=\sum_{\substack{n\leq x\\P^+(n)>x^{1-c}\\P^+(n-1)>P^+(n+1)}}(U_n-V_n),\\&Q_c:=\sum_{\substack{n\leq x\\P^+(n)>x^{1-c}\\P^+(n-1)<P^+(n+1)}}(V_n-U_n).
	\end{align*}
	It is not difficult to see (also by Lemma 2.4)
	\begin{align*}
		&P_c-Q_c=\sum_{\substack{n\leq x\\P^+(n)>x^{1-c}}}(U_n-V_n)=o(x),\\
		&Q_c\leq \log (x+1)\mathcal{B}_1.
	\end{align*}
	For $P_c$, we have
	\begin{align*}
		P_c\geq \sum_{\substack{n\leq x,n\notin\mathcal{C}_{c,\delta}\\P^+(n)>x^{1-c}\\P^+(n-1)>P^+(n+1)}}(U_n-V_n)\geq \delta\log x\mathcal{B}_2.
	\end{align*}
	Thus 
	\begin{align*}
		\mathcal{B}_2\leq \frac{1}{\delta}\mathcal{B}_1+o(x).\tag{3.2}
	\end{align*}
	By (3.1), we have
	\begin{align*}
		\frac{1+\delta}{\delta}\mathcal{B}_1+o(x)\geq \mathcal{B}_1+\mathcal{B}_2\geq  \frac{x}{4}\log\frac{1}{1-c},
	\end{align*}
	i.e.,
	\begin{align*}
		\sum_{\substack{n\leq x\\P^+(n)>x^{1-c}\\P^+(n-1)<P^+(n+1)}}1\geq x\frac{\delta}{4(1+\delta)}\log\frac{1}{1-c}.
	\end{align*}
	For the convenience of our subsequent discussion, we summarize the above as follows: There exists $\eta>0$ and $C_1=C_1(\eta)>0$ such that for any $0<c<\eta$ one has
	\begin{align*}
		\sum_{\substack{n\leq x\\P^+(n)>x^{1-c}\\P^+(n-1)<P^+(n+1)}}1\geq xC_1\log\frac{1}{1-c}.\tag{3.3}
	\end{align*}
	\subsection{Proof of Theorem 1.3}
	If $P^+(n+1)>x^{1-c}$, we have the following fact that (by Lemma 2.2, also see [16, p.369-373])
	\begin{align*}
		\sum_{\substack{n\leq x\\\min(P^+(n),P^+(n+1))>x^{1-c}}}1\leq  C_2c^2x.
	\end{align*}
	Thus we get
	\begin{align*}
		\sum_{\substack{n\leq x\\P^+(n-1)<P^+(n+1)<P^+(n)}}1&\geq \sum_{\substack{n\leq x\\P^+(n)>x^{1-c}\\P^+(n-1)<P^+(n+1)}}1-\sum_{\substack{n\leq x\\\min(P^+(n),P^+(n+1))>x^{1-c}}}1\\&\geq \left(C_1\log\frac{1}{1-c}-C_2c^2\right)x.\tag{3.4}
	\end{align*}
	Let $g(c):=-C_1\log(1-c)-C_2c^2$. We have $g(0)=0$ and $g'(0)=C_1> 0$. Then there exists $c_0\in(0.\eta)$ such that $g(c_0)>0$, i.e.,
	\begin{align*}
		\sum_{\substack{n\leq x\\P^+(n-1)<P^+(n+1)<P^+(n)}}1\gg x.
	\end{align*}
	Similarly, we also have
	\begin{align*}
		\sum_{\substack{n\leq x\\P^+(n+1)>P^+(n-1)<P^+(n)}}\gg x.
	\end{align*}
	If we add a further series, we get
	\begin{align*}
		\sum_{\substack{n\leq x\\P^+(n-1)<P^+(n+1)<P^+(n)>\max_{i}P^+(n+a_i)}}1\geq& \sum_{\substack{n\leq x\\P^+(n)>x^{1-c}\\P^+(n-1)<P^+(n+1)}}1-\sum_{\substack{n\leq x\\\min(P^+(n),P^+(n+1))>x^{1-c}}}1\\&-\sum_i\sum_{\substack{n\leq x\\\min(P^+(n),P^+(n+a_i))>x^{1-c}}}1\\\geq& \left(C_1\log\frac{1}{1-c}-(i+1)C_2c^2\right)x,\tag{3.5}
	\end{align*}
	where the $a_i$ are at most finite in number, with $|a_i| = O(1)$.
	Let $g_i(c):=-C_1\log(1-c)-(i+1)C_2c^2$. We have $g_i(0)=0$ and $g_i'(0)=C_1> 0$. Then there exists $c_{i0}\in(0,\eta)$ such that $g_i(c_{i0})>0$, i.e.,
	\begin{align*}
		\sum_{\substack{n\leq x\\P^+(n-1)<P^+(n+1)<P^+(n)>\max_{i}P^+(n+a_i)}}1\gg x.\tag{3.6}
	\end{align*}
	Similarly, we also have
	\begin{align*}
		\sum_{\substack{n\leq x\\P^+(n+1)>P^+(n-1)<P^+(n)>\max_{i}P^+(n+a_i)}}1\gg x.\tag{3.7}
	\end{align*}
	By adjusting the constant appropriately, we complete the proof.
	\begin{remark}
		In this section, we replace integers with large prime factors by smooth numbers. Similar to Proposition 3.1, we have the following result. For every $\varepsilon>0$, for any $0<1/u<1/C$ with the $C$ in Lemma 2.7, there exists $\delta=\delta(u,\varepsilon)>0$ such that the set $\mathcal{C}_{u,\delta}$ defined by
			\begin{align*}
				&\mathcal{C}_{u,\delta}':=\\
				&\#\left\{n\leq x:P^+(n)<x^{1/u}, P^+(n-1)>x^{1/2},P^+(n+1)>x^{1/2},x^{-\delta}<\frac{P^+(n-1)}{P^+(n+1)}<x^\delta\right\}
			\end{align*}
			satisfies
			\begin{align*}
				\mathcal{C}_{u,\delta}'\leq\varepsilon \rho(u)x.
			\end{align*}
		Unfortunately, it seems we cannot make $\delta$ independent of the choice of $u$. Thus we seem unable to obtain
		\begin{align*}
			\sum_{\substack{n\leq x\\P^+(n)>P^+(n+1)<P^+(n+2)<P^+(n+3)}}1\gg x.
		\end{align*}
	\end{remark}
	
	\section{Proof of Theorem 1.5}
	We only present the proof for the first case, and the remaining cases may be analyzed by similar arguments.
	For any $0<\alpha<1/10<9/10<\beta<1$, we have
	\begin{align*}
		&\sum_{\substack{n\leq x\\P^+(n)>P^+(n+1)<P^+(n+2)<P^+(n+3)>P^+(n+4)}}1\\&\geq \sum_{\substack{n\leq x\\P^+(n+1)\leq x^\alpha, P^+(n+3)>x^{\beta}}}1-\sum_{\substack{n\leq x\\P^+(n+1)\leq x^\alpha, P^+(n+3)>x^{\beta}\\P^+(n+2)\geq x^\beta\ \text{or}\ P^+(n+4)\geq x^\beta}}1-\sum_{\substack{n\leq x\\\\P^+(n+1)\leq x^\alpha, P^+(n+3)>x^{\beta}\\P^+(n)\leq x^\alpha\ \text{or}\ P^+(n+2)\leq x^\alpha}}1\\&:=\mathcal{S}_1-\mathcal{S}_2-\mathcal{S}_3.\tag{4.1}
	\end{align*}
	For the first summation, under the Elliott-Halberstam conjecture for friable integers, we have (we can see it in [17])
	\begin{align*}
		\mathcal{S}_1\sim \rho\left(\frac{1}{\alpha}\right)\left(1-\rho\left(\frac{1}{\beta}\right)\right)x.\tag{4.2}
	\end{align*}
	For the second summation, we have
	\begin{align*}
		\mathcal{S}_2\leq \sum_{\substack{n\leq x\\P^+(n+1)\leq x^\alpha, P^+(n+3)>x^{\beta}\\P^+(n+2)\geq x^\beta}}1+\sum_{\substack{n\leq x\\P^+(n+1)\leq x^\alpha, P^+(n+3)>x^{\beta}\\P^+(n+4)\geq x^\beta}}1.
	\end{align*}
	We rewrite $n+3=ap_1$ and $n+2=bp_2$ with $\min(p_1,p_2)>x^\beta$. Then we have (by Lemma 2.3. Also see [11, proof of Theorem 1.9] for the one-dimensional case; in fact, we may simplify its proof.)
	\begin{align*}
		\sum_{\substack{n\leq x\\P^+(n+1)\leq x^\alpha, P^+(n+3)>x^{\beta}\\P^+(n+2)\geq x^\beta}}1&\leq \sum_{1\leq a,b<x^{1-\beta}}\sum_{\substack{n\leq x,P^+(n)\leq x^\alpha\\a|(n+2),b|(n+1)\\(n+1)/b,(n+2)/a \ \text{are primes}}}1+o(x)\\&\ll (1-\beta)^2\rho\left(\frac{1}{\alpha}\right)x+o(x),
	\end{align*}
	then
	\begin{align*}
		\mathcal{S}_2\ll (1-\beta)^2\rho\left(\frac{1}{\alpha}\right)x+o(x)\ll \rho\left(\frac{1}{\alpha}\right)\left(1-\rho\left(\frac{1}{\beta}\right)\right)(1-\beta)x+o(x).\tag{4.3}
	\end{align*}
	For the third summation, we have
	\begin{align*}
		\mathcal{S}_3\leq \sum_{\substack{n\leq x\\P^+(n+1)\leq x^\alpha, P^+(n+3)>x^{\beta}\\P^+(n)\leq x^\alpha}}1+\sum_{\substack{n\leq x\\P^+(n+1)\leq x^\alpha, P^+(n+3)>x^{\beta}\\P^+(n+2)\leq x^\alpha}}1.
	\end{align*}
	Note that
	\begin{align*}
		\log n=\sum_{m|n}\Lambda(m).
	\end{align*}
	Consider $x/\log x<n\leq x$. Then we have
	\begin{align*}
		&\sum_{\substack{n\leq x\\P^+(n+1)\leq x^\alpha, P^+(n+3)>x^{\beta}\\P^+(n)\leq x^\alpha}}1\\&=\frac{1}{\log x}\sum_{\substack{n\leq x\\P^+(n+1)\leq x^\alpha, P^+(n+3)>x^{\beta}\\P^+(n)\leq x^\alpha}}\log n+o(x)\\&=\frac{1}{\log x}\sum_{p'\leq x^\alpha}\log p'\sum_{\substack{n\leq x\\P^+(n+1)\leq x^\alpha, P^+(n+3)>x^{\beta}\\p'|n}}1+O\Bigg(\frac{1}{\log x}\sum_{\substack{p',k\\p'^k\leq x,k\geq 2}}\log p'\sum_{\substack{n\leq x\\p'^k|n}}1\Bigg)+o(x).
	\end{align*}
	For the error term, we have
	\begin{align*}
		\frac{1}{\log x}\sum_{\substack{p',k\\p'^k\leq x,k\geq 2}}\log p'\sum_{\substack{n\leq x\\p'^k|n}}1\ll \frac{x}{\log x}\sum_{p'\leq x^{1/2}}\log p'\sum_{k\geq 2}\frac{1}{p'^k}\ll \frac{x}{\log x}\sum_{p'\leq x^{1/2}}\frac{\log p'}{p'^2}=O\left(\frac{x}{\log x}\right).
	\end{align*}
	Write $n+3=ap_1$ with $p_1>x^\beta$. For the main term, we have (by Lemma 2.2)
	\begin{align*}
		&\frac{1}{\log x}\sum_{p'\leq x^\alpha}\log p'\sum_{\substack{n\leq x\\P^+(n+1)\leq x^\alpha, P^+(n+3)>x^{\beta}\\p'|n}}1\\&\ll\frac{1}{\log x}\sum_{p'\leq x^\alpha}\log p'\sum_{a<x^{1-\beta}}\sum_{\substack{n\leq x,P^+(n)\leq x^\alpha\\ p'|(n-1), a|(n+2)\\(n+2)/a\ \text{is prime}}}1\\&\ll \alpha(1-\beta)\rho\left(\frac{1}{\alpha}\right)x+o(x),
	\end{align*}
	then
	\begin{align*}
		\mathcal{S}_3\ll \alpha(1-\beta)\rho\left(\frac{1}{\alpha}\right)x+o(x)\ll \alpha\rho\left(\frac{1}{\alpha}\right)\left(1-\rho\left(\frac{1}{\beta}\right)\right)x+o(x).\tag{4.4}
	\end{align*}
	To summarize the above, for sufficiently large $C>0$, we have
	\begin{align*}
		\sum_{\substack{n\leq x\\P^+(n)>P^+(n+1)<P^+(n+2)<P^+(n+3)>P^+(n+4)}}1\geq g(\alpha,\beta)x+o(x),
	\end{align*}
	where $0<\alpha<1/10<9/10<\beta<1$,
	\begin{align*}
		g(\alpha,\beta)=\rho\left(\frac{1}{\alpha}\right)\left(1-\rho\left(\frac{1}{\beta}\right)\right)\Bigg(1-C(1-\beta)-C\alpha\Bigg).
	\end{align*}
	Let $\alpha:=1/(3C)$ and $\beta:=1-1/(3C)$. Then we have $g(\alpha,\beta)>0$, i.e., 
	\begin{align*}
		\sum_{\substack{n\leq x\\P^+(n)>P^+(n+1)<P^+(n+2)<P^+(n+3)>P^+(n+4)}}1\gg x.\tag{4.5}
	\end{align*}
	
	\section{Proof of Theorem 1.7}
	We only show the first case: $1< P^+(n+1)/P^+(n)\leq1+\varepsilon$. Let $\delta=\varepsilon/2$.
	Write $n=ap_1$ and $n+1=bp_2$, with $P<p_1\leq (1+\delta)P$ and $(1+\delta)P<p_2<(1+2\delta)P$.
	Then we have
	\begin{align*}
		1<\frac{p_2}{p_1}<1+2\delta.
	\end{align*}
	Since $p_1$ and $p_2$ are not necessarily the corresponding largest prime factors, we need to remove some terms. We have
	\begin{align*}
		\sum_{\substack{n\leq x\\1<P^+(n+1)/P^+(n)<1+2\delta}}&1\geq \sum_{P}\Bigg(\sum_{(1+\delta)P<p_2<(1+2\delta)P}\sum_{P<p_1\leq (1+\delta)P}\sum_{\substack{a\leq x/p_1\\ap_1\equiv -1\mkern-15mu\pmod{p_2}}}1\\&-\sum_{P<p_1\leq (1+\delta)P}\sum_{(1+\delta)P<p_2<(1+2\delta)P}\Bigg(\sum_{\substack{a\leq x/p_1\\P^+(a)>P\\ap_1\equiv -1\mkern-15mu\pmod{p_2}}}1+\sum_{\substack{b\leq x/p_2\\P^+(b)>(1+\delta)P\\bp_2\equiv 1\mkern-15mu\pmod{p_1}}}1\Bigg)\Bigg)\\=&\sum_{P}(S_1(P)-S_2(P)),\tag{5.1}
	\end{align*}
	where $\sum_P$ is taken over all  $P=x^{1/2-\delta}(1+\delta)^{-i}$, $1\leq i\leq \delta\log x/\log (1+\delta)\asymp \log x$. We have 
	\begin{align*}
		S_1(P)&=\sum_{(1+\delta)P<p_2<(1+2\delta)P}\sum_{P<p_1\leq (1+\delta)P}\frac{x}{p_1p_2}+o(x^{1-\delta})\\&\geq C_1\delta^2\frac{x}{(\log x)^2}+o(x^{1-\delta}),\tag{5.2}
	\end{align*}
	since
	\begin{align*}
		&\sum_{(1+\delta)P<p_2<(1+2\delta)P}\sum_{P<p_1\leq (1+\delta)P}\frac{x}{p_1p_2}\\&\gg\frac{x}{P^2}\left(\pi((1+2\delta)P)-\pi((1+\delta)P)\right)\left(\pi((1+\delta)P)-\pi(P)\right)\\&\gg \delta^2\frac{x}{(\log x)^2}.
	\end{align*}
	For $S_2$, by writing $a=a_0p_0$ with $p_0>P$, we have
	\begin{align*}
		&\sum_{P<p_1\leq (1+\delta)P}\sum_{(1+\delta)P<p_2<(1+2\delta)P}\sum_{\substack{a\leq x/p_1\\P^+(a)>P\\ap_1\equiv -1\mkern-15mu\pmod{p_2}}}1\\&=\sum_{(1+\delta)P<p_2<(1+2\delta)P}\frac{1}{\varphi(p_2)}\sum_{P<p_1\leq (1+\delta)P}\sum_{\substack{a\leq x/p_1\\P^+(a)>P\\(ap_1,p_2)=1}}1\\&\quad+\sum_{(1+\delta)P<p_2<(1+2\delta)P}\Bigg(\sum_{P<p_1\leq (1+\delta)P}\sum_{\substack{a\leq x/p_1\\P^+(a)>P\\ap_1\equiv -1\mkern-15mu\pmod{p_2}}}1-\frac{1}{\varphi(p_2)}\sum_{P<p_1\leq (1+\delta)P}\sum_{\substack{a\leq x/p_1\\P^+(a)>P\\(ap_1,p_2)=1}}1\Bigg).
	\end{align*}
	For the remainder term, by Lemma 2.6, the contribution is negligible. (To do this, we first put $a_0,p_0,p_1$ in the dynamic ranges $[\mathscr{A}_0,(1+\Delta)\mathscr{A}_0]$, $[\mathscr{P}_0,(1+\Delta)\mathscr{P}_0]$, $[\mathscr{P}_1,(1+\Delta)\mathscr{P}_1]$ respectively, where $\Delta=(\log x)^{-A_0}$ for a suitable constant $A_0$. Then let $q\leftarrow p_2$, $m\leftarrow a_0p_1$, $n\leftarrow p_0$, $\beta_n\leftarrow\mathbf{1}_{p_0\in[\mathscr{P}_0,(1+\Delta)\mathscr{P}_0]}$ and $\alpha_m\leftarrow1_{a_0\in[\mathscr{A}_0,(1+\Delta)\mathscr{A}_0]}1_{p_1\in(\mathscr{P}_1,(1+\Delta)\mathscr{P}_1]}$.) For the main term, we have
	\begin{align*}
		&\sum_{(1+\delta)P<p_2<(1+2\delta)P}\frac{1}{\varphi(p_2)}\sum_{P<p_1\leq (1+\delta)P}\sum_{\substack{a\leq x/p_1\\P^+(a)>P\\(ap_1,p_2)=1}}1\\&=\sum_{(1+\delta)P<p_2<(1+2\delta)P}\frac{1}{\varphi(p_2)}\sum_{P<p_1\leq (1+\delta)P}\sum_{P<p_0\leq x/p_1}\frac{x}{p_0p_1}\\&\ll \delta\frac{x}{P^2}\left(\pi((1+2\delta)P)-\pi((1+\delta)P)\right)\left(\pi((1+\delta)P)-\pi(P)\right)\\&\ll \delta^3\frac{x}{(\log x)^2}.
	\end{align*}
	Thus we have 
	\begin{align*}
		S_2(P)\leq C_2\delta^3\frac{x}{(\log x)^2}+O\left(\frac{x}{(\log x)^{10}}\right).\tag{5.3}
	\end{align*}
	Returning to (5.1), we have
	\begin{align*}
		\sum_{\substack{n\leq x\\1<P^+(n+1)/P^+(n)<1+2\delta}}1\geq \delta^2\left(C_1-C_2\delta\right)\frac{x}{\log x}+O\left(\frac{x}{(\log x)^{9}}\right).\tag{5.4}
	\end{align*}
	Considering $\delta<C_1/(2C_2)$, we complete the proof.
	
	\begin{remark}
		For the upper bound, we use similar arguments for $\max(P^+(n),P^+(n+1))<x^{1/2}$, and use sieve method for $\min(P^+(n),P^+(n+1))>x^{1/2+\varepsilon}$. The remaining ranges appear to be crucial for the estimate.
	\end{remark}

	\section*{Acknowledgements}
	The authors thank Professor Zhiwei Wang for exchanging some information.

\end{document}